\documentclass[english,reqno]{amsart}
\usepackage[T1]{fontenc}
\usepackage[latin9]{inputenc}
\usepackage{amsthm}
\usepackage{amstext}
\usepackage{amssymb}

\makeatletter
\numberwithin{equation}{section}
\numberwithin{figure}{section}
\theoremstyle{plain}
\newtheorem{thm}{Theorem}[section]
  \theoremstyle{definition}
  \newtheorem{defn}[thm]{Definition}
  \theoremstyle{remark}
  \newtheorem{rem}[thm]{Remark}
  \theoremstyle{plain}
  \newtheorem{prop}[thm]{Proposition}
  \theoremstyle{plain}
  \newtheorem{lem}[thm]{Lemma}
  \theoremstyle{plain}
  \newtheorem{cor}[thm]{Corollary}
  \theoremstyle{plain}
  \newtheorem{fact}[thm]{Fact}

\usepackage{amscd}

\theoremstyle{plain}

\makeatother

\usepackage{babel}

\begin{document}

\title[Interpolation of cocompact imbeddings]{On interpolation of cocompact imbeddings}

\author{Michael Cwikel}

\address{Department of Mathematics, Technion - Israel Institute of Technology,
Haifa 32000, Israel}

\email{mcwikel@math.technion.ac.il}

\urladdr{http://www.math.technion.ac.il/\textasciitilde{}mcwikel/}

\author{Kyril Tintarev}

\address{Department of Mathematics, Uppsala University, P.O.Box 480, 75 106
Uppsala, Sweden }

\email{tintarev@math.uu.se}

\urladdr{http://tintarev.comxa.com}

\thanks{The research was supported by the Technion V.P.R.\ Fund and by the
Fund for Promotion of Research at the Technion. Parts of this research
have been conducted during visits of the authors to each other's home
institutions.}


\begin{abstract}
Cocompactness is a useful weaker counterpart of compactness in the
study of imbeddings between function spaces. In this paper we prove
that, under quite general conditions, cocompactness of imbeddings
of Banach spaces persists under both real and complex interpolation.
As an application, we obtain that subcritical continuous imbeddings
of fractional Sobolev spaces and Besov spaces are cocompact relative
to lattice shifts. We deduce this by interpolating the known cocompact
imbeddings for classical Sobolev spaces (``vanishing''
lemmas of Lieb and Lions). We also apply cocompactness to prove compactness
of imbeddings of some radial subspaces and to show the existence of
minimizers in some isoperimetric problems. Our research complements
a range of previous results, and recalls that there is a natural conceptual
framework for unifying them. 
\end{abstract}
\maketitle
\smallskip{}


\section{\label{intro}Introduction.}

The notion of \emph{cocompact imbedding} is a convenient way to
express a property of imbeddings, related to (and weaker than) compactness.
Several authors, including Elliott Lieb, Pierre-Louis Lions and Terence
Tao, have proved and used cocompactness of imbeddings of Sobolev and
Strichartz spaces into $L^{p}$-spaces, without explicitly using this
terminology. Although these and other results have long been perceived
to be related to each other in the heuristic sense of \emph{concentration
compactness}, a formal and unified interpretation of them in functional
analytic terms appeared only relatively recently.

Our starting point is the classical definition of a cocompact manifold:
a manifold $M$ is called cocompact relative to a given group $G$
of automorphisms $\eta:M\to M$ if there exists a compact subset $K\Subset M$
such that for each $x\in M$ there exists some $\eta\in G$ such that
$\eta x\in K$. In particular this implies that, for any sequence
$\left\{ x_{k}\right\} _{k\in\mathbb{N}}$ in $M$, there exists a
sequence $\left\{ \eta_{k}\right\} _{k\in\mathbb{N}}$ in $G$ such
that $\left\{ \eta_{k}x_{k}\right\} _{k\in\mathbb{N}}$ has a convergent
subsequence in $M$.

When we seek to introduce a related notion in the context of Banach
spaces, it turns out to be natural to do this in terms of the following
modified version of weak convergence, defined relative to some group
of continuous linear bijections. 
\begin{defn}
\label{def:dwc}Let $A$ be a Banach space, and let $\mathcal{D}$
be a group of continuous linear bijections of $A$. A sequence $\left\{ u_{k}\right\} _{k\in\mathbb{N}}$
of elements of $A$ is said to be $\mathcal{D}$\textit{-weakly convergent
}to $u\in A$ (denoted by writing $u_{k}\stackrel{\mathcal{D}}{\rightharpoonup}u$
), if $g_{k}(u_{k}-u)\rightharpoonup0$ for all choices of the sequence
$\left\{ g_{k}\right\} _{k\in\mathbb{N}}$ in $\mathcal{D}$.
\end{defn}
Note that since $\mathcal{D}$ contains the identity operator $I$,
any $\mathcal{D}$-weakly convergent sequence is also weakly convergent.
The converse is true if $\mathcal{D}$ is finite.

We can now state our main definition:
\begin{defn}
\label{def:cocompact imbedding}Let $A$ and $B$ be Banach spaces
such that $A$ is continuously imbedded into $B$. Let $\mathcal{D}$
be a group of continuous linear bijections on $A$. We say that the
imbedding of $A$ into $B$ is \textit{cocompact relative to} $\mathcal{D}$
if every $\mathcal{D}$-weakly convergent sequence $\left\{ u_{k}\right\} _{k\in\mathbb{N}}$
in $A$ converges in $B$.\end{defn}
\begin{rem}
There is a slightly different version of Definition \ref{def:dwc}
in \cite{Fieseler-Tintarev} where $\mathcal{D}$ is merely a set,
not necessarily a group, and all its elements are bounded linear operators.
Definition \ref{def:cocompact imbedding} first appeared in \cite{Tintarev-BanachCC},
in fact in a marginally more restricted version where $\mathcal{D}$
is a group of linear isometries. 
\end{rem}
For later use we record the following elementary result which follows
immediately from Definition \ref{def:cocompact imbedding}.

\begin{prop}
\label{pro:3spaces}Let $X_{1}$, $X_{2}$ and $X_{3}$ be three Banach
spaces with continuous imbeddings $X_{1}\subset X_{2}\subset X_{3}$.
Suppose that the group $\mathcal{D}$ of linear operators $g:X_{3}\to X_{3}$
acts isometrically from $X_{j}$ to $X_{j}$ for $j=1,2,3$. Then
the imbedding $X_{1}\subset X_{3}$ is $\mathcal{D}$-cocompact whenever
at least one of the imbeddings $X_{1}\subset X_{2}$ and $X_{2}\subset X_{3}$
is $\mathcal{D}$-cocompact. \end{prop}
\begin{rem}
When{, in the context of Definition \ref{def:cocompact imbedding},}
$\mathcal{D}=\left\{ I\right\} $, {and $A$
is reflexive,} then of course cocompact imbeddings are simply compact
imbeddings. 
\end{rem}

\begin{rem}
The notion of cocompactness facilitates the formulation of many results
which can be considered as extensions of the Banach-Alaoglu theorem
and which have wide applications. For example, an abstract version
of such a result, in the context of Hilbert spaces due to Schindler
and the second author of this paper is given in \cite{SchindlerTintarev}
and also appears as Theorem 3.1 of \cite{Fieseler-Tintarev} pp.~62--63.
Other versions, most of them in the setting of particular function
spaces, and some of them for particular sequences, have been obtained
(often independently) by a number of authors. Various terminologies,
such as \textit{splitting lemmas, profile decompositions, intermediate
topology, défaut de compacité, mutual divergence, dislocations }or\textit{
rescalings, vanishing sequences, asymptotic orthogonality,} etc.~have
been used by a number of authors to describe the phenomena encountered
in these results. We refer to \cite{TinSurvey} for a survey of such
results and their applications.

The abstract Hilbert space version (i.e. the above-mentioned result
in \cite{SchindlerTintarev}) of this refinement of the Banach-Alaoglu
theorem states that, in the presence of a suitably chosen group of
operators $\mathcal{D}$ acting on a Hilbert space $A$, every bounded
subsequence in $A$ has a subsequence of the following special structure:
Each term in the subsequence is the sum of a principal term and a
remainder term. The remainder terms form a sequence which converges
$\mathcal{D}$-weakly, and each principal term is a (possibly infinite)
sum of {}``dislocated profiles'', i.e. terms of the form $g_{k}w$
where $g_{k}\in\mathcal{D}$ and $w\in A$. 

Although no general analogue of this result is known for the case
where $A$ is an arbitrary Banach space, the other results alluded
to at the beginning of this remark are all of the same form, for suitable
particular choices of $A$ and $\mathcal{D}$. The above-mentioned
term, profile decomposition, is the one most commonly used for the
special subsequence provided by such results.

The practical value of profile decompositions depends on finding some
concrete space $B$ 
{where the remainder sequence
tends to zero in the norm.}
This is precisely what can be assured in those cases where the
imbedding of $A$ into $B$ is known to be cocompact.
\end{rem}
In an early cocompactness result about classical Sobolev imbeddings,
$\mathcal{D}$ is taken to be the group $\mathcal{D}_{\mathbb{R}^{N}}$
of shifts $u\mapsto u(\cdot-y)$. This result is essentially due to
Lieb \cite{Lieb} Lemma 6 on p.~447. In fact Lieb showed that any
$\mathcal{D}_{\mathbb{R}^{N}}$-weakly convergent sequence in $W^{1,p}\left(\mathbb{R}^{N}\right)$
converges in measure, from which one can easily conclude that subcritical
imbeddings of $W^{1,p}\left(\mathbb{R}^{N}\right)$ are $\mathcal{D}_{\mathbb{R}^{N}}$-cocompact.
The first explicit statement of this latter result is due, independently,
to P.-L.~Lions \cite{PLL1b} Lemma I.1 on p.~231, which we restate
below as Theorem \ref{thm:lieb}. In the paper \cite{Lions87} Lions
showed the existence of a profile decomposition for a specific sequence
in $W^{1,p}\left(\mathbb{R}^{N}\right)$, again with $\mathcal{D}=\mathcal{D}_{\mathbb{R}^{N}}$.
This result gave a more detailed description of the {}``loss of compactness''
for that sequence than had been shown in his celebrated papers \cite{PLL1a,PLL1b,PLL2a,PLL2b}
on concentration compactness. The first proof of the existence of
a profile decomposition for an \textit{arbitrary} bounded sequence
in the homogeneous Sobolev space $\dot{W}^{1,p}\left(\mathbb{R}^{N}\right)$
was given by Solimini \cite{Solimini} for the case where $\mathcal{D}$
is the product group of the actions of translations and dilations
on $\mathbb{R}^{N}$. It is easy to deduce the existence of the particular
profile decomposition in \cite{Lions87} from the result of \cite{Solimini}.
Subsequently Gérard \cite{Gerard} gave an independent proof of a
slightly weaker version of Solimini's result for the same group $\mathcal{D}$,
but for somewhat different function spaces, namely $\dot{W}^{s,p}(\mathbb{R}^{N})$
with $0<s<N/p$ and $p=2$. Gérard's result was extended by Jaffard
\cite{Jaffard} to all $p\in(1,\infty)$. Independently, the second
author of this paper obtained similar results 
(Chapter 9\cite{Fieseler-Tintarev})
some portion of them jointly with Fieseler or Schindler, where $\mathbb{R}^{N}$
is replaced by a cocompact Riemannian manifold, an arbitrary nilpotent
stratified Lie group, or a fractal blowup, with appropriate choices
of $\mathcal{D}$. Additional references to other related works will
be given in \cite{TinSurvey}. Among the most recent developments
we mention the papers \cite{Tao1,Tao2}, in which Tao, Visan and Zhang
have proved the cocompactness of a Strichartz imbedding for the time-dependent
Schrödinger equation, and also the work of Koch \cite{Koch} where
results in the style of \cite{Gerard,Jaffard} are presented for embeddings
of $L^{p}$ spaces into homogeneous Besov spaces with negative index
of smoothness.

The main result of this paper, Theorem \ref{thm:main abstract}, deals
with persistence of cocompactness for interpolated spaces. It can
be considered as a sort of counterpart to results about persistence
of compactness for operators mapping between {}``real method'' or
{}``complex method'' interpolation spaces, in particular those in
Section 9.6 of \cite{ca} and in \cite{Persson}, in which hypotheses
having a partial analogy with hypotheses of Theorem \ref{thm:main abstract}
are imposed. 
\begin{rem}
Note however that the compactness results of \cite{ca} and \cite{Persson}
were subsequently found to 
also hold
without these kinds of hypotheses
and/or under other alternative hypotheses. (See e.g.~\cite{CKS,CwRCI,CwKa,CwJa,CwLattice}
and the references therein.) An analogous complete removal of additional
conditions in the case of cocompactness would mean that persistence
of cocompactness under interpolation holds for all choices of the
group $\mathcal{D}$, which remains an open question. A negative answer
to it would not surprise us.
\end{rem}
As examples of applications of Theorem \ref{thm:main abstract}, we
prove the cocompactness of classical Peetre imbeddings \cite{PeetreSoboleff}
of inhomogeneous Sobolev spaces with fractional indices of smoothness
into $L^{p}$spaces, relative to the group 
$\mathcal{D}_{{\mathbb{Z}^{N}}}$
of {lattice} shifts $u\mapsto u(\cdot-y)$ {with
$y\in\mathbb{Z}^{N}$}. 
This is done in Theorem \ref{thm:ournew}.
Analogous results for imbeddings of Besov spaces are given in Theorems
\ref{thm:cocoBesov} and \ref{thm:Coco-Bes-Lp}. These latter results
can be thought of as variants of the results mentioned above of Gérard,
Jaffard and Koch. In some ways they are not as sharp. On the other
hand, unlike their results, ours deal with the case of inhomogeneous
spaces.

Our results are stated in Section \ref{sec:mainresults}. Almost all
of their proofs are deferred to subsequent sections. In Section 3
we give the proof of Theorem \ref{thm:main abstract}. Proofs of Theorem
\ref{thm:ournew} and of Theorems \ref{thm:cocoBesov} and \ref{thm:Coco-Bes-Lp}
are provided in Sections 4 and 5 respectively. We conclude our paper
with Section 6 presenting two applications: We prove the compactness
of certain imbeddings for subspaces of radial functions and the existence
of minimizers for some isoperimetric problems involving fractional
Sobolev spaces.

This paper lies at the intersection of two fields, the analysis of
Sobolev spaces, and the theory of interpolation spaces. Since some
readers may be more familiar with one of these fields than the other,
we have taken the liberty of summarizing some of the basic notions
from both of them, partly in the main body of the paper, and partly
in appendices. In particular, in Appendix A, we recall definitions
and results which we need from interpolation space theory, and in
Appendix B we provide a version of the Brezis-Lieb lemma.

\section{\label{sec:mainresults}Statements of the main results}

In all that follows, whenever we deal with Banach spaces, whose elements
are functions $u:\mathbb{R}^{N}\to\mathbb{\mathbb{C}}$ and whose
norms are translation invariant, we will always choose the group $\mathcal{D}$
of Definitions \ref{def:dwc}and \ref{def:cocompact imbedding} to
be the set of lattice shifts. In other words, we take

\begin{equation}
\mathcal{D}=\mathcal{D}_{\mathbb{Z}^{N}}:=\left\{ g_{y}\right\} _{y\in\mathbb{Z}^{N}}\ \mbox{where }g_{y}u=u(\cdot-y)\,.\label{eq:defdlat}\end{equation}

Whenever we deal here with a Banach couple $\left(A_{0},A_{1}\right)$
we will always associate a group $\mathcal{D}$ to that couple, and
the elements $g$ of $\mathcal{D}$ will always be assumed to be linear
operators $g:$$A_{0}+A_{1}\to A_{0}+A_{1}$, such that \begin{equation}
g(A_{j})\subset A_{j}\ \mbox{and }g:A_{j}\to A_{j}\ \mbox{is an isometry for }j=0,1\,.\label{eq:gauge}\end{equation}

There are several frequently used different equivalent norms for the
interpolation spaces $\left(A_{0},A_{1}\right)_{\theta,p}$, $\left[A_{0},A_{1}\right]_{\theta}$
and for $A_{0}+A_{1}$. Here we will always use the standard norms
whose definitions are recalled in Appendix A. 
\begin{lem}
Let $(A_{0},A_{1})$ be a Banach couple and let $\mathcal{D}$ be
a group of linear maps $g:A_{0}+A_{1}\to A_{0}+A_{1}$ satisfying
(\ref{eq:gauge}). Then each $g\in\mathcal{D}$ is also an isometry
on $A_{0}+A_{1}$. Moreover, for every $p\in[1,\infty)$, $\theta\in(0,1)$,
the restriction of $g$ to $(A_{0},A_{1})_{\theta,p}$, respectively
$[A_{0},A_{1}]_{\theta}$, is an isometry on $(A_{0},A_{1})_{\theta,p}$,
respectively $[A_{0},A_{1}]_{\theta}$.\end{lem}
\begin{proof}
This follows immediately from the basic interpolation properties of
the spaces $\left(A_{0},A_{1}\right)_{\theta,p}$ and $\left[A_{0},A_{1}\right]_{\theta}$
and $A_{0}+A_{1}$ applied for the operators $g$ and $g^{-1}$. 
\end{proof}
\smallskip{}

We now introduce a definition of an operator family whose properties
(i) and (ii) below are reminiscent of various conditions imposed to
obtain interpolation of compactness in Section 9.6 of \cite{ca} and
in \cite{Persson}. As we shall see below, the standard mollifiers
in Sobolev spaces, equipped with lattice shifts, are an example of
a family of operators $M_{t}$ satisfying the definition.
\begin{defn}
\label{def:mollifiers}Let $\left(A_{0},A_{1}\right)$ be a Banach
couples
with $A_{1}$ 
is
continuously imbedded in $A_{0}$ and let
$\mathcal{D}$ be a group of linear operators $g:A_{0}+A_{1}\to A_{0}+A_{1}$
which satisfies (\ref{eq:gauge}). 
Let $A_{1}$be continuously
imbedded
into some Banach space $B_{1}$. A family of bounded operators $\{M_{t}\}_{t\in(0,1)}$
from $A_{0}$ to $A_{1}$ is said to be a \emph{family of }$\mathcal{D}$\emph{-covariant}
\emph{mollifiers} (relative to a space $B_{1}$) if it satisfies the
following conditions: 

\noindent \begin{flushleft}
$\begin{array}{llc}
\mbox{(i)} & \mbox{For }j=0,1,\:\mbox{the norm of }M_{t}\,\mbox{as a continuous map from }A_{j}\,\mbox{into itself is }\\
 & \mbox{bounded independently of }t\in(0,1),\,\mbox{i.e.,}{\displaystyle {\displaystyle }\sup_{t\in(0,1)}}\Vert M_{t}\Vert_{A_{j}\to A_{j}}<\infty\,.\\
\\\mbox{\mbox{(ii)}} & \mbox{The function }\sigma(t):=\|I-M_{t}\|_{A_{1}\to B_{1}}\text{ satisfies }{\displaystyle \lim_{t\to0}\,}\sigma(t)=0\,.\\
\\\mbox{(iii)} & \mbox{For each }g\in\mathcal{D},\;\mbox{and }t\in(0,1),\ \mbox{there exists an element }h_{g,t}\in\mathcal{D\ }\\
 & \mbox{\mbox{such that }}gM_{t}=M_{t}h_{g,t}\,.\end{array}$
\par\end{flushleft}
\end{defn}
Our main result is expressed in terms of general Banach couples. 
\begin{thm}
\label{thm:main abstract}Let $\left(A_{0},A_{1}\right)$ and $\left(B_{0},B_{1}\right)$
be Banach couples with $A_{j}$ continuously imbedded in $B_{j}$
for $j=0,1$. Suppose, further, that $A_{1}$ is continuously imbedded
in $A_{0}$. Let $\mathcal{D}$ be a group of linear operators $g:B_{0}+B_{1}\to B_{0}+B_{1}$
which satisfies (\ref{eq:gauge}) with respect to both of the couples
$\left(A_{0},A_{1}\right)$ and $\left(B_{0},B_{1}\right)$. Assume
that there exists a $\mathcal{D}$-covariant mollifier family $\{M_{t}:\, A_{0}\to A_{1}\}_{t\in(0,1)}$.
(See Definition \ref{def:mollifiers}.) If, furthermore, $A_{1}$
is $\mathcal{D}$-cocompactly imbedded into $B_{1}$, then, for every
$\theta\in(0,1)$ and $q\in[1,\infty]$, the space $(A_{0},A_{1})_{\theta,q}$
is $\mathcal{D}$-cocompactly imbedded into $(B_{0},B_{1})_{\theta,q}$
and the space $[A_{0},A_{1}]_{\theta}$ is $\mathcal{D}$-cocompactly
imbedded into $[B_{0},B_{1}]_{\theta}$. 
\end{thm}
\smallskip{}

We shall apply Theorem \ref{thm:main abstract} to obtain cocompactness
of interpolated imbeddings between certain function spaces. Our point
of departure for doing this is the following cocompactness property
of Sobolev imbeddings. It can be immediately shown to be an equivalent
reformulation of Lemma 6 on p.~447 of Lieb's paper \cite{Lieb} and
also of Lemma I.1 on p.~231 of Lion's paper \cite{PLL1b}. 
\begin{thm}
\label{thm:lieb} Suppose that $p\in(1,\infty)$. The Sobolev imbedding
of $W^{1,p}(\mathbb{R}^{N})$ into $L^{q}(\mathbb{R}^{N})$, $p<q<p^{*}$,
where $p^{*}=\frac{pN}{N-p}$ for $N>p$ and $p^{*}=\infty$ otherwise,
is $\mathcal{D}_{\mathbb{Z}^{N}}$-cocompact.
\end{thm}
In the following elementary application of Theorem \ref{thm:main abstract},
we shall extend this property to the Sobolev imbedding of the spaces
$W^{\alpha,p}(\mathbb{R}^{N})$ for all $\alpha\in(0,\infty)$. We
recall one of the equivalent definitions of the space $W^{\alpha,p}(\mathbb{R}^{N})$,
namely as the space of all functions $f:\mathbb{R}^{N}\to\mathbb{R}$
in $L^{p}(\mathbb{R}^{N})$ whose Fourier transforms $\widehat{f}$
are such that $(1+\left|\xi\right|^{2})^{\alpha/2}\widehat{f}(\xi)$
is also the Fourier transform of a function in $L^{p}(\mathbb{R}^{N})$.
This definition is valid for all real values of $\alpha>0$, including
non integer values.

We recall the Sobolev--Peetre imbedding theorem, which states that
the continuous inclusion $W^{\alpha,p}(\mathbb{R}^{N})\subset L^{q}(\mathbb{R}^{N})$
holds whenever $\alpha$ is positive and $1<p\le q\le p_{\alpha}^{*}$,
where the critical exponent $p_{\alpha}^{*}$ is defined by \begin{equation}
p_{\alpha}^{*}=\left\{ \begin{array}{ccc}
\frac{pN}{N-\alpha p} & , & N>\alpha p\\
\infty & , & N\le\alpha p\end{array}\right..\label{eq:dce}\end{equation}
When $\alpha=1$ the prevalent notation is to write $p^{*}$ instead
of $p_{1}^{*}$ (as we did just above in Theorem \ref{thm:lieb}). 
\begin{thm}
\label{thm:ournew}Suppose that $\alpha\in(0,\infty)$ and $p\in(1,\infty)$.
The Sobolev--Peetre imbedding of $W^{\alpha,p}(\mathbb{R}^{N})$ into
$L^{q}(\mathbb{R}^{N})$ is $\mathcal{D}_{\mathbb{Z}^{N}}$-cocompact
whenever $p<q<p_{\alpha}^{*}$. Moreover, the imbedding $W^{\alpha+\gamma,p}(\mathbb{R}^{N})\subset W^{\gamma,q}(\mathbb{R}^{N})$
is $\mathcal{D}_{\mathbb{Z}^{N}}$-cocompact for every $\gamma>0$.
\end{thm}
We now state our third result, which is obtained by applying Theorem
\ref{thm:main abstract} to couples of Sobolev spaces, for which the
real interpolation method yields Besov spaces. (Relevant definitions
are recalled in Appendix A.) The continuity of the imbeddings considered
in this theorem is due to Jawerth \cite{Jawerth}.
\begin{thm}
\label{thm:cocoBesov}Suppose that $0<\beta<\alpha<\infty$ and $1<p_{0}<p_{1}<\infty$
and $q\in[1,\infty]$. If $\frac{N}{p_{0}}-\frac{N}{p_{1}}<\alpha-\beta$,
then the continuous imbedding of $B^{\alpha,p_{0},q}(\mathbb{R}^{N})$
into $B^{\alpha,p_{1},q}(\mathbb{R}^{N})$ is $\mathcal{D}_{\mathbb{Z}^{N}}$-cocompact.\end{thm}
\begin{cor}
\label{cor:cocoBesov}Let $\alpha$, $\beta$, $p_{0}$, $p_{1}$
and $N$ be as in Theorem \ref{thm:cocoBesov}. Then the imbedding
of $B^{\alpha,p_{0},q_{0}}(\mathbb{R}^{N})$ into $B^{\beta,p_{1},q_{1}}(\mathbb{R}^{N})$
is $\mathcal{D}_{\mathbb{Z}^{N}}$-cocompact whenever $1\le q_{0}\le q_{1}\le\infty$. 
\end{cor}
This corollary follows immediately from Proposition \ref{pro:3spaces}.
We take $X_{1}=B^{\alpha,p_{0},q_{0}}$, $X_{2}=B^{\beta,p_{1},q_{0}}$
and $X_{3}=B^{\beta,p_{1},q_{1}}$. By Theorem \ref{thm:cocoBesov},
$X_{1}$ is $\mathcal{D}_{\mathbb{Z}^{N}}$-cocompactly imbedded into
$X_{2}$. The continuous imbedding $X_{2}\subset X_{3}$ follows from
(\ref{eq:besovbi}) and (\ref{eq:monotonicity}). 
\begin{thm}
\label{thm:Coco-Bes-Lp}Let $s>0,\;1<p<\infty,\: p<q_{0}\le q<p_{s}^{*}$.
Then the imbedding of $B^{s,p,q_{0}}(\mathbb{R}^{N})$ into $L^{q}(\mathbb{R}^{N})$
is $\mathcal{D}_{\mathbb{Z}^{N}}$-cocompact.
\end{thm}

\section{The proof of Theorem \ref{thm:main abstract}}

We consider the case of real interpolation. The proof for the complex
case is completely analogous. 

In view of the continuous imbedding $\left(A_{0},A_{1}\right)_{\theta,q}\subset A_{0}+A_{1}=A_{0}$,
it follows that, for each fixed $t$, the operator $M_{t}$ is bounded
from $\left(A_{0},A_{1}\right)_{\theta,q}$ into $A_{1}$. Suppose
that $u_{k}\stackrel{\mathcal{D}}{\rightharpoonup}0$ in $(A_{0},A_{1})_{\theta,q}$.
Let $\left\{ g_{k}\right\} _{k\in\mathbb{N}}$ be an arbitrary sequence
in $\mathcal{D}$. Then \begin{equation}
g_{k}M_{t}u_{k}=M_{t}h_{g_{k},t}u_{k}\label{eq:wn}\end{equation}
by property (iii). Since $h_{g_{k},t}u_{k}\rightharpoonup0$ in $\left(A_{0},A_{1}\right)_{\theta,q}$,
we deduce that $M_{t}h_{g_{k},t}u_{k}\rightharpoonup0$ in $A_{1}$
for each fixed $t\in(0,1)$. The cocompactness of the imbedding $A_{1}\subset B_{1}$
and (\ref{eq:wn}) now imply that \begin{equation}
\lim_{k\to\infty}\Vert M_{t}u_{k}\Vert_{B_{1}}=0\,.\label{eq:2zero}\end{equation}

In view of the continuous inclusions $A_{j}\subset B_{j}$ and property
(i), we have that $M_{t}:A_{j}\to B_{j}$ is bounded with \begin{equation}
S_{j}:=\sup_{t\in(0,1)}\Vert M_{t}\Vert_{A_{j}\to B_{j}}<\infty\,,\ \mbox{for }j=0,1\,.\label{eq:bothbdd}\end{equation}

Since $M_{t}u_{k}\in B_{0}\cap B_{1}$, we can invoke (\ref{eq:drb})
in Appendix A and then (\ref{eq:bothbdd}) to obtain that \begin{eqnarray*}
\Vert M_{t}u_{k}\Vert_{\left(B_{0},B_{1}\right)_{\theta,q}} & \le & c_{\theta,q}\Vert M_{t}u_{k}\Vert_{B_{0}}^{1-\theta}\Vert M_{t}u_{k}\Vert_{B_{1}}^{\theta}\\
 & \le & c_{\theta,q}\left(S_{0}\Vert u_{k}\Vert_{A_{0}}\right)^{1-\theta}\Vert M_{t}u_{k}\Vert_{B_{1}}^{\theta}\,.\end{eqnarray*}
Since $\left\{ u_{k}\right\} _{k\in\mathbb{N}}$ is necessarily a
bounded sequence in the space $\left(A_{0},A_{1}\right)_{\theta,q}$
(by the Banach-Steinhaus Theorem), and is therefore also bounded in
the space $A_{0}$, we can use (\ref{eq:2zero}) to obtain that \begin{equation}
\lim_{k\to\infty}\Vert M_{t}u_{k}\Vert_{\left(B_{0},B_{1}\right)_{\theta,q}}=0\,.\label{eq:ftz}\end{equation}

We now consider the operator $I-M_{t}$ in more detail. By (\ref{eq:bothbdd})
we of course have $I-M_{t}:A_{0}\to B_{0}$ with $\left\Vert I-M_{t}\right\Vert _{A_{0}\to B_{0}}\le\left\Vert I\right\Vert _{A_{0}\to B_{0}}+S_{0}$.
Using this estimate, property (ii) and Theorem \ref{thm:fund interp},
we obtain that $I-M_{t}$ is a bounded operator from $\left(A_{0},A_{1}\right)_{\theta,q}$
into $\left(B_{0},B_{1}\right)_{\theta,q}$ and that \begin{eqnarray*}
\Vert I-M_{t}\Vert_{\left(A_{0},A_{1}\right)_{\theta,q}\to\left(B_{0},B_{1}\right)_{\theta,q}} & \le & \Vert I-M_{t}\Vert_{A_{0}\to B_{0}}^{1-\theta}\Vert I-M_{t}\Vert_{A_{1}\to B_{1}}^{\theta}\\
 & \le & \left(\Vert I\Vert_{A_{0}\to B_{0}}+S_{0}\right)^{1-\theta}\sigma(t)^{\theta}\,.\end{eqnarray*}
Therefore, with the help of (\ref{eq:ftz}), we have \begin{eqnarray*}
\limsup_{k\to\infty}\Vert u_{k}\Vert_{\left(B_{0},B_{1}\right)_{\theta,q}} & \le & \limsup_{k\to\infty}\Vert M_{t}u_{k}\Vert_{\left(B_{0},B_{1}\right)_{\theta,q}}+\limsup_{k\to\infty}\Vert(I-M_{t})u_{k}\Vert_{\left(B_{0},B_{1}\right)_{\theta,q}}\\
 & \le & 0+\limsup_{k\to\infty}\Vert(I-M_{t})u_{k}\Vert_{\left(B_{0},B_{1}\right)_{\theta,q}}\\
 & \le & \limsup_{k\to\infty}\left(\Vert I\Vert_{A_{0}\to B_{0}}+S_{0}\right)^{1-\theta}\sigma(t)^{\theta}\Vert u_{k}\Vert_{\left(A_{0},A_{1}\right)_{\theta,q}}\,.\end{eqnarray*}
We now use the boundedness of the sequence $\left\{ \Vert u_{k}\Vert_{\left(A_{0},A_{1}\right)_{\theta,q}}\right\} _{k\in\mathbb{N}}$
once more, together with property (ii), to obtain that this last expression
is bounded by a quantity which tends to $0$ as $t$ tends to $0$.
Since we can choose $t$ as small as we please, this shows that $\lim_{k\to\infty}\Vert u_{k}\Vert_{\left(B_{0},B_{1}\right)_{\theta,q}}=0$
and completes the proof of the theorem.$\qed$

\section{\label{sec:scc}Cocompactness of the imbedding $W^{\alpha,p}\subset L^{q}$
for all $\alpha\in(0,\infty)$}

In this section we give the proof of Theorem \ref{thm:ournew}.

Let $\Lambda$ be the operator $I-\Delta$ which of course corresponds
to the Fourier multiplier $1+\left|\xi\right|^{2}$. Note that $\Lambda$
commutes with all of the operators $g\in\mathcal{D}_{\mathbb{Z}^{N}}$,
as does each of its powers. Furthermore, $\Lambda^{\gamma/2}$ defines
an isometry between $W^{\alpha+\gamma,p}(\mathbb{R}^{N})$ and $W^{\alpha,p}(\mathbb{R}^{N})$
as well as between $W^{\gamma,q}(\mathbb{R}^{N})$ and $L^{q}(\mathbb{R}^{N})$.
Since $\mathcal{D}_{\mathbb{Z}^{N}}$-weak convergence is preserved
by each of these isometries and their inverses, we see that the $\mathcal{D}_{\mathbb{Z}^{N}}$-cocompactness
of the imbedding $W^{\alpha+\gamma,p}(\mathbb{R}^{N})\subset W^{\gamma,q}(\mathbb{R}^{N})$
is an immediate consequence of the $\mathcal{D}_{\mathbb{Z}^{N}}$-cocompactness
of the imbedding $W^{\alpha,p}(\mathbb{R}^{N})\subset L^{q}(\mathbb{R}^{N})$
which we will now prove.

We begin by considering the case where $\alpha\in(0,1)$. Here we
apply Theorem \ref{thm:main abstract} to suitable Banach couples
of $L^{p}$ and Sobolev spaces. 

We present the first step as the following lemma. Here, as before,
$p^{*}=p_{1}^{*}$ is the critical exponent defined in (\ref{eq:dce}).
\begin{lem}
\label{lem:verified}Let $\left(A_{0},A_{1}\right)=\left(L^{p}(\mathbb{R}^{N}),W^{1,p}(\mathbb{R}^{N})\right)$
and $(B_{0},B_{1})=\left(L^{p}(\mathbb{R}^{N}),L^{r}(\mathbb{R}^{N})\right)$
with $r\in(p,p^{*})$. Let $\mathcal{D}=\mathcal{D}_{\mathbb{Z}^{N}}$. 

Let $\rho:\mathbb{R}^{N}\to[0,\infty)$ be a $C^{\infty}$ function
with support contained in the open unit ball $\left\{ z\in\mathbb{R}^{N}:\left|z\right|<1\right\} $
which satisfies $\int_{\mathbb{R^{N}}}\rho(x)dx=1$. 

Then, for each fixed $t\in(0,1)$ the operator $M_{t}$, which is
defined by \begin{equation}
\left(M_{t}u\right)(x)=\int_{|z|<1}\rho(z)u(x+tz)dz\,,\label{eq:M_t}\end{equation}
is a bounded map of $A_{0}$ into $A_{1}$, and the family $\left\{ M_{t}\right\} _{t\in(0,1)}$
satisfies properties (i), (ii) and (iii) of Definition \ref{def:mollifiers}.
\end{lem}
\textit{Proof.} The boundedness of $M_{t}$ from $A_{0}$ into $A_{1}$
for each fixed $t$ is simply the well known mollification property.
It is also obvious that $M_{t}:A_{j}\to A_{j}$ is bounded with $\left\Vert M_{t}\right\Vert _{A_{j}\to A_{j}}\le1$
for $j=0,1$ and all $t\in(0,1)$, which gives property (i). 

Property (iii) is an immediate consequence of the fact that $(M_{t}u)(\cdot-y)=M_{t}(u(\cdot-y))$
for each $y\in\mathbb{R}^{N}$. In fact here we can take $h_{g,t}=g$
for each $g\in\mathcal{D}_{\mathbb{Z}^{N}}$ and each $t$. 

It remains to prove property (ii). Consider the following identity:
\[
u(x)-M_{t}u(x)=\int_{|z|<1}\rho(z)[u(x)-u(x+tz)]dz=-\int_{|z|<1}\rho(z)\int_{0}^{t}z\cdot\nabla u(x+sz)dsdz.\]

Then \[
|u(x)-M_{t}u(x)|^{p}\le\sup_{\left|y\right|<1}\rho(y)^{p}\left|\int_{0}^{t}\int_{|z|<1}|\nabla u(x+sz)|dz\, ds\right|^{p}.\]

By Hölder's inequality we then have \[
|u(x)-M_{t}u(x)|^{p}\le Ct^{p/p'}\int_{0}^{t}\int_{|z|<1}|\nabla u(x+sz)|^{p}dz\, ds.\]

Integrating with respect to $x$, we obtain \begin{eqnarray}
\int_{\mathbb{R}^{N}}|u(x)-M_{t}u(x)|^{p}dx & \le & Ct^{p/p'}\int_{0}^{t}\int_{|z|<1}\int_{\mathbb{R}^{N}}|\nabla u(x+sz)|^{p}dx\, dz\, ds\nonumber \\
 & = & Ct^{p/p'}\int_{0}^{t}\int_{|z|<1}\int_{\mathbb{R}^{N}}|\nabla u(x)|^{p}dx\, dz\, ds\nonumber \\
 & = & Ct^{1+p/p'}\int_{\mathbb{R}^{N}}|\nabla u(x)|^{p}dx\,.\label{eq:nn}\end{eqnarray}

Here, and also later, we will use the following immediate consequence
of Hölder's inequality:
\begin{fact}
\label{fac:holder}The inclusion $L^{p_{0}}\cap L^{p_{1}}\subset L^{p}$
holds whenever $1\le p_{0}<p<p_{1}\le\infty$. Furthermore, the estimate
\begin{equation}
\left\Vert f\right\Vert _{p}\le\left\Vert f\right\Vert _{p_{0}}^{1-\theta}\left\Vert f\right\Vert _{p_{1}}^{\theta}\label{eq:hold}\end{equation}
holds for each $f\in L^{p_{0}}\cap L^{p_{1}}$, where $\theta=\Theta(p_{0},p,p_{1}):={\displaystyle \frac{\frac{1}{p_{0}}-\frac{1}{p}}{\frac{1}{p_{0}}-\frac{1}{p_{1}}}}\in(0,1)$.
\end{fact}
Let $s$ be some number satisfying $r<s<p^{*}$. Then $p<r<s$ and
so Fact \ref{fac:holder} gives us that \begin{equation}
\|u-M_{t}u\|_{r}\le\|u-M_{t}u\|_{p}^{1-\theta}\|u-M_{t}u\|_{s}^{\theta}\,,\ \mbox{where }\theta=\frac{\frac{1}{p}-\frac{1}{r}}{\frac{1}{p}-\frac{1}{s}}\in(0,1)\,.\label{eq:hsb}\end{equation}
We estimate $\left\Vert u-M_{t}u\right\Vert _{p}$ and $\left\Vert u-M_{t}u\right\Vert _{s}$
using, respectively, (\ref{eq:nn}) and the Sobolev imbedding theorem.
Substituting these estimates in (\ref{eq:hsb}), and noting that $1+p/p'=p$,
we obtain that \begin{eqnarray*}
\|u-M_{t}u\|_{r} & \le & C\left(t^{p}\left\Vert u\right\Vert _{W^{1,p}}\right)^{1-\theta}\left(\left\Vert u-M_{t}u\right\Vert _{W^{1,p}}\right)^{\theta}\\
 & \le & Ct^{(1-\theta)p}\left\Vert u\right\Vert _{W^{1,p}}\,.\end{eqnarray*}
This establishes property (ii) and completes the proof of the lemma.
\hfill{}$\qed$

We will now prove the assertion of Theorem \ref{thm:ournew} for $\alpha\in(0,1)$
and for some particular value of $q\in(p,p_{\alpha}^{*})$. 

For the number $p\in(1,\infty)$ appearing in the statement of Theorem
\ref{thm:ournew}, and for some number $r$ in $(p,p^{*})$ we let
$\left(A_{0},A_{1}\right)$ and $\left(B_{0},B_{1}\right)$ be the
same couples $\left(L^{p}(\mathbb{R}^{N}),W^{1,p}(\mathbb{R}^{N})\right)$
and $\left(L^{p}(\mathbb{R}^{N}),L^{r}(\mathbb{R}^{N})\right)$ which
appear in Lemma \ref{lem:verified}. Let us also choose the group
$\mathcal{D}$ and the family of operators $\left\{ M_{t}\right\} _{t\in(0,1)}$
to be as in Lemma \ref{lem:verified}. 

We know, using Theorem \ref{thm:lieb}, that that $A_{1}$ is $\mathcal{D}_{\mathbb{Z}^{N}}$-cocompactly
imbedded in $B_{1}$. This, together with Lemma \ref{lem:verified},
provides us with all the conditions required for applying Theorem
\ref{thm:main abstract} in this context. More specifically, if we
invoke the statement about complex interpolation spaces in Theorem
\ref{thm:main abstract}, we obtain that $[L^{p}(\mathbb{R}^{N}),W^{1,p}(\mathbb{R}^{N})]_{\theta}$
is $\mathcal{D}_{\mathbb{Z}^{N}}$-cocompactly imbedded in $[L^{p}(\mathbb{R}^{N}),L^{r}(\mathbb{R}^{N})]_{\theta}$
for each $\theta\in(0,1)$. By standard results (see Appendix A),
these two spaces are $W^{\theta,p}(\mathbb{R}^{N})$ and $L^{s_{0}}(\mathbb{R}^{N})$
respectively, where $s_{0}$ is the number in the interval $(p,r)$
given by \begin{equation}
\frac{1}{s_{0}}=\frac{1-\theta}{p}+\frac{\theta}{r}\,.\label{eq:spr}\end{equation}

Setting $\theta=\alpha$, we see that this establishes our result
for $q=s_{0}$. It will now be easy to extend the proof to all $q\in(p,p_{\alpha}^{*})$:

Let $\left\{ u_{k}\right\} _{k\in\mathbb{N}}$ be an arbitrary sequence
in $W^{\alpha,p}$ which converges $\mathcal{D}_{\mathbb{Z}^{N}}$-weakly
to $0$.

Given an arbitrary $q$ in $\left(p,p_{\alpha}^{*}\right)$ we choose
$r\in(p,p^{*})$ sufficiently close to $p$ so that the number $s_{0}$
given by (\ref{eq:spr}), with $\theta=\alpha$, satisfies $p<s_{0}<q$.
By the previous step of our argument we also have that $\lim_{k\to\infty}\left\Vert u_{k}\right\Vert _{L^{s_{0}}(\mathbb{R}^{N})}=0$.
Now let us choose some number $s_{1}\in(q,p_{\alpha}^{*})$. By the
Sobolev imbedding theorem, the sequence $\left\{ u_{k}\right\} _{k\in\mathbb{N}}$,
which is bounded in $W^{1,p}(\mathbb{R}^{N})$, must also be bounded
in $L^{s_{1}}(\mathbb{R}^{N})$. Finally, we use Fact \ref{fac:holder}
to bound $\left\Vert u_{k}\right\Vert _{q}$ by $\left\Vert u_{k}\right\Vert _{s_{0}}^{1-\beta}\left\Vert u_{k}\right\Vert _{s_{1}}^{\beta}$
for a suitable number $\beta\in(0,1)$. This suffices to complete
the proof of Theorem \ref{thm:ournew} for the case $\alpha\in(0,1)$. 

The case where $\alpha=1$ is of course Theorem \ref{thm:lieb}. So
it remains to deal with the easy case where $\alpha>1$. 

Let $p$ and $q$ be as in the statement of the theorem. Noting that
we always have $p<p^{*}$, let us choose numbers $q_{0}$ and $q_{1}$
which satisfy $p<q_{0}<\min\left\{ p^{*},q\right\} $ and $q<q_{1}<p_{\alpha}^{*}$.
Consider an arbitrary sequence $\left\{ u_{k}\right\} _{k\in\mathbb{N}}$
in $W^{\alpha,p}\left(\mathbb{R}^{N}\right)$ which is $\mathcal{D}_{\mathbb{Z}^{N}}$-weakly
convergent to zero. Since in this case $W^{\alpha,p}(\mathbb{R}^{N})$
is continuously imbedded into $W^{1,p}(\mathbb{R}^{N})$, we have
that $u_{k}(\cdot-y_{k})\rightharpoonup0$ in $W^{1,p}(\mathbb{R}^{N})$
for any sequence $\left\{ y_{k}\right\} _{k\in\mathbb{N}}$ of elements
of $\mathbb{Z}^{N}$, i.e., $u_{k}$ is $\mathcal{D}_{\mathbb{Z}^{N}}$-weakly
convergent in $W^{1,p}(\mathbb{R}^{N})$. Then, by Theorem \ref{thm:lieb},
$\lim_{k\to\infty}\left\Vert u_{k}\right\Vert _{q_{0}}=0$. 

Since $q_{0}<q<q_{1}$, Fact \ref{fac:holder} gives us that 

\begin{equation}
\|u_{k}\|_{q}\le\|u_{k}\|_{q_{0}}^{1-\theta}\|u_{k}\|_{q_{1}}^{\theta}\,,\ \mbox{where }\theta=\frac{\frac{1}{q_{0}}-\frac{1}{q}}{\frac{1}{q_{0}}-\frac{1}{q_{1}}}\in(0,1)\,.\label{eq:holderstuff}\end{equation}
Then, since $W^{\alpha,p}(\mathbb{R}^{N})$ is continuously imbedded
into $L^{q_{1}}(\mathbb{R}^{N})$, we have
$\|u_{k}\|_{q}\le C\|u_{k}\|_{q_{0}}^{1-\theta}\|u_{k}\|_{W^{\alpha,p}}^{\theta}$.
Since (again by the Banach-Steinhaus Theorem) weakly convergent sequences
are bounded, we obtain that $\|u_{k}\|_{q}\le C\|u_{k}\|_{q_{0}}^{\theta}\to0$.
\hfill{}$\qed$

\section{Cocompact imbeddings of Besov spaces and the proofs of Theorems \ref{thm:cocoBesov}
and \ref{thm:Coco-Bes-Lp}}

The following lemma will be the main component of the proof of Theorem
\ref{thm:cocoBesov}.
\begin{lem}
\label{lem:sigma}Suppose that $m_{0},m_{1}\in\mathbb{R}$, $0\le m_{1}<m_{0}$,
$1<p_{0}<p_{1}<\infty$, and assume further that \begin{equation}
\frac{1}{p_{0}}-\frac{1}{p_{1}}<\frac{m_{0}-m_{1}}{N}\,.\label{eq:copo}\end{equation}
For each $t\in(0,1)$, the operator $M_{t}$ defined by (\ref{eq:M_t})
is a bounded map from $W^{m_{0},p_{0}}(\mathbb{R}^{N})$ to $W^{m_{1},p_{1}}(\mathbb{R}^{N})$
and satisfies 

\begin{equation}
\lim_{t\to0}\|I-M_{t}\|_{W^{m_{0},p_{0}}(\mathbb{R}^{N})\to W^{m_{1},p_{1}}(\mathbb{R}^{N})}=0\,.\label{eq:sigma}\end{equation}
\end{lem}
\begin{proof}
We begin by observing that the conditions on $p_{0}$ and $p_{1}$
in the statement of the lemma are equivalent to \begin{equation}
1<p_{0}<p_{1}<(p_{0})_{m_{0}-m_{1}}^{*}\,.\label{eq:ecop}\end{equation}
(The notation here is as specified in (\ref{eq:dce}), and this equivalence
holds whether or not $(p_{0})_{m_{0}}^{*}$ is finite.). 

We shall make use once more of the operator $\Lambda=I-\Delta$ which
was introduced at the beginning of Section \ref{sec:scc}, noting
that $\Lambda$ and each of its powers all commute with all of the
operators $M_{t}$. 
{Since $\Lambda^{{m_{1}}/2}$
defines an isometry between $W^{{m_{1},p_{1}}}(\mathbb{R}^{N})$
and $L^{{p_{1}}}(\mathbb{R}^{N})$ as well as between 
$W^{{m_{0},p_{0}}}(\mathbb{R}^{N})$
and $W^{{m_{0}-m_{1},p_{0}}}(\mathbb{R}^{N})$,}
it suffices to prove the lemma in the case where the two parameters
$m_{0}$ and $m_{1}$ are replaced by $m_{0}^{\prime}=m_{0}-m_{1}$
and $m_{1}^{\prime}=m_{1}-m_{1}=0$, i.e. we can suppose that $m_{1}=0$.
Note that this {}``shift'' of the values of $m_{0}$ and $m_{1}$
does not change the stated conditions on $p_{0}$ and $p_{1}$. 

Case 1: Assume first that $m_{0}\ge1$. By Lemma \ref{lem:verified},
we have \begin{equation}
\lim_{t\to0}\|I-M_{t}\|_{W^{1,p_{0}}\to L^{r}}=0\ \mbox{for each }r\in\left(p_{0},(p_{0})^{*}\right)\,.\label{eq:oep}\end{equation}
This also implies that \begin{equation}
\lim_{t\to0}\|I-M_{t}\|_{W^{m_{0},p_{0}}\to L^{r}}=0\,.\label{eq:okfr}\end{equation}

Sub-case 1.1: If $p_{1}=r$ the lemma is proved. 

Sub-case 1.2: If $p_{1}>r$ then we can obtain (\ref{eq:sigma}) by
using Fact \ref{fac:holder} with $r$, $p_{1}$ and some number $s\in\left(p_{1},(p_{0})_{m_{0}}^{*}\right)$
now assuming the roles of $p_{0}$, $p$ and $p_{1}$ respectively.
More precisely, for each $f\in W^{m_{0},p_{0}}$ and for $\theta=\Theta(r,p_{1},s)\in(0,1)$,
we have \begin{eqnarray}
\left\Vert (I-M_{t})f\right\Vert _{L^{p_{1}}} & \le & \left\Vert (I-M_{t})f\right\Vert _{L^{r}}^{1-\theta}\left\Vert (I-M_{t})f\right\Vert _{L^{s}}^{\theta}\nonumber \\
 & \le & \left(\left\Vert I-M_{t}\right\Vert _{W^{m_{0},p_{0}}\to L^{r}}\left\Vert f\right\Vert _{W^{m_{0},p_{0}}}\right)^{1-\theta}\left(2\left\Vert f\right\Vert _{L^{s}}\right)^{\theta}\,.\label{eq:tmji}\end{eqnarray}
Since $p_{0}<s<(p_{0})_{m_{0}}^{*}$ we have that $\left\Vert f\right\Vert _{L^{s}}$
is bounded by a constant multiple of $\left\Vert f\right\Vert _{W^{m_{0},p_{0}}}$
which we can substitute in (\ref{eq:epg}) and then use (\ref{eq:okfr})
to obtain the required property (\ref{eq:sigma}) in this case. 

Sub-case 1.3: If $p_{1}<r$, we use an argument similar to the one
for Sub-case 1.2. This time we apply Fact \ref{fac:holder} with $p_{0},$
$p_{1}$ and $r$ in the roles of $p_{0}$, $p$ and $p_{1}$. Accordingly,
analogously to (\ref{eq:eptmji}), for each $f\in W^{m_{0},p_{0}}$
and for $\theta=\Theta(p_{0},p_{1},r)\in(0,1)$, we have that \begin{eqnarray}
\left\Vert (I-M_{t})f\right\Vert _{L^{p_{1}}} & \le & \left\Vert (I-M_{t})f\right\Vert _{L^{p_{0}}}^{1-\theta}\left\Vert (I-M_{t})f\right\Vert _{L^{r}}^{\theta}\nonumber \\
 & \le & \left(2\left\Vert f\right\Vert _{L^{p_{0}}}\right)^{1-\theta}\left(\left\Vert I-M_{t}\right\Vert _{W^{m_{0},p_{0}}\to L^{r}}\left\Vert f\right\Vert _{W^{m_{0},p_{0}}}\right)^{\theta}\,.\label{eq:epg}\end{eqnarray}
Obviously $\left\Vert f\right\Vert _{L^{p_{0}}}\le\left\Vert f\right\Vert _{W^{m_{0},p_{0}}}$
and so the proof is also complete in this case. 

Case 2: If $0<m_{0}<1$, then we apply Theorem \ref{thm:fund interp}
to the operator $T=I-M_{t}$ and the couples $\left(A_{0},A_{1}\right)=\left(L^{p_{0}},W^{1,p_{0}}\right)$
and $\left(B_{0},B_{1}\right)=(L^{p_{0}},L^{r})$ where $r\in\left(p_{0},(p_{0})^{*}\right)$.
We choose $\theta=m_{0}$ and use the facts (see Appendix A) that
$W^{m_{0},p_{0}}=\left[L^{p_{0}},W^{1,p_{0}}\right]_{m_{0}}$ and
$\left[L^{p_{0}},L^{r}\right]_{m_{0}}=L^{s_{0}}$, where \begin{equation}
\frac{1}{s_{0}}=\frac{1-m_{0}}{p_{0}}+\frac{m_{0}}{r}\,.\label{eq:psm}\end{equation}
 Thus we obtain that\begin{eqnarray}
\left\Vert I-M_{t}\right\Vert _{W^{m_{0},p_{0}}\to L^{s_{0}}} & \le & \left\Vert I-M_{t}\right\Vert _{L^{p_{0}}\to L^{p_{0}}}^{1-m_{0}}\left\Vert I-M_{t}\right\Vert _{W^{1,p_{0}}\to L^{r}}^{m_{0}}\nonumber \\
 & \le & 2^{1-m_{0}}\left\Vert I-M_{t}\right\Vert _{W^{1,p_{0}}\to L^{r}}^{m_{0}}\,.\label{eq:iwcs}\end{eqnarray}

Since we are free to choose $r$ arbitrarily close to $p_{0}$, we
see from (\ref{eq:psm}) that we can also have $s_{0}$ arbitrarily
close to $p_{0}$. So, keeping (\ref{eq:ecop}) in mind, let us choose
$r$ so that $s_{0}<p_{1}$ and let us choose a second number $s_{1}\in\left(p_{1},(p_{0})_{m_{0}}^{*}\right)$.
Now we use Fact \ref{fac:holder} once more: For each $f\in W^{m_{0},p_{0}}$,
and for $\theta=\Theta(s_{0},p_{1},s_{1})\in(0,1)$, we have \begin{eqnarray}
\left\Vert (I-M_{t})f\right\Vert _{L^{p_{1}}} & \le & \left\Vert (I-M_{t})f\right\Vert _{L^{s_{0}}}^{1-\theta}\left\Vert (I-M_{t})f\right\Vert _{L^{s_{1}}}^{\theta}\nonumber \\
 & \le & \left(\left\Vert I-M_{t}\right\Vert _{W^{m_{0},p_{0}}\to L^{s_{0}}}\left\Vert f\right\Vert _{W^{m_{0},p_{0}}}\right)^{1-\theta}\left(2\left\Vert f\right\Vert _{L^{s_{1}}}\right)^{\theta}\,.\label{eq:eptmji}\end{eqnarray}
The fact that $s_{1}\in\left(p_{0},(p_{0})_{m_{0}}^{*}\right)$ ensures
that $\left\Vert f\right\Vert _{L^{s_{1}}}$ is bounded by a constant
multiple of $\left\Vert f\right\Vert _{W^{m_{0},p_{0}}}$. After we
substitute this in (\ref{eq:eptmji}) and apply (\ref{eq:iwcs}) and
then (\ref{eq:oep}), we obtain (\ref{eq:sigma}) in this final case,
and so complete the proof of the lemma.
\end{proof}
After these preparations, the proof of Theorem \ref{thm:cocoBesov}
is almost immediate. Let $\epsilon\in(0,\beta/2)$ and let $\alpha_{0}=\alpha+\epsilon$,
$\alpha_{1}=\alpha-\epsilon$, $\beta_{0}=\beta+\epsilon$ and $\beta_{1}=\beta-\epsilon$.
Consider the Banach couples 

\[
\left(A_{0},A_{1}\right)=\left(W^{\alpha_{0},p_{0}}(\mathbb{R}^{N}),W^{\alpha_{1},p_{0}}(\mathbb{R}^{N})\right)\mbox{ and }\left(B_{0},B_{1}\right)=\left(W^{\beta_{0},p_{1}}(\mathbb{R}^{N}),W^{\beta_{1},p_{1}}(\mathbb{R}^{N})\right)\,.\]

Let $\lambda=\frac{N}{p_{0}}-\frac{N}{p_{1}}$. For $j=0,1$, since
$\alpha_{j}-\beta_{j}=\alpha-\beta>\lambda$, we obtain from Theorem
\ref{thm:ournew}, that $A_{j}$ is $\mathcal{D}_{\mathbb{Z}^{N}}$-cocompactly
imbedded in $B_{j}$. This, together with Lemma \ref{lem:sigma},
shows that the conditions for applying Theorem \ref{thm:main abstract}
are fulfilled. So we can deduce that $(A_{0},A_{1})_{\theta,q}$ is
$\mathcal{D}_{\mathbb{Z}^{N}}$-cocompactly imbedded into $(B_{0},B_{1})_{\theta,q}$
for each $\theta\in(0,1)$ and $q\in[1,\infty]$. In particular, if
we choose $\theta=1/2$ we obtain the assertion of the theorem. $\qed$
\medskip{}

We now turn to the proof of Theorem \ref{thm:Coco-Bes-Lp}. Obviously
in view of (\ref{eq:besovbi}), (\ref{eq:monotonicity}) and Proposition
\ref{pro:3spaces}, it suffices to consider the case where $q_{0}=q$.
Fix some $\theta\in(0,1)$ and define $s_{0}$ and $r$ so that they
satisfy $s=\theta s_{0}$ and \begin{equation}
\frac{1}{q}=\frac{1-\theta}{p}+\frac{\theta}{r}\,.\label{eq:gian}\end{equation}
We next want to show that \begin{equation}
q<r<p_{s_{0}}^{*}\,.\label{eq:nwts}\end{equation}
The first inequality of (\ref{eq:nwts}) follows from (\ref{eq:gian})
and the fact that $p<q$. The second inequality of (\ref{eq:nwts})
is equivalent to \[
\frac{1}{r}>\frac{1}{p}-\frac{s_{0}}{N}\,,\]
which readily follows from $1/q>1/p-s/N=1/p-\theta s_{0}/N$ and (\ref{eq:gian}).

In view of (\ref{eq:nwts}) and Theorem \ref{thm:ournew} we have
that $W^{s_{0},p}(\mathbb{R}^{N})$ is $\mathcal{D}_{\mathbb{Z}^{N}}$-cocompactly
imbedded into $L^{r}(\mathbb{R}^{N})$. Then, by Theorem \ref{thm:main abstract}
it follows that the imbedding \[
\left(L^{p},W^{s_{0},p}\right)_{\theta,r}\subset\left(L^{p},L^{r_{0}}\right)_{\theta,r}\]
 is $\mathcal{D}_{\mathbb{Z}^{N}}$-cocompact. Using (\ref{eq:besovbi})
and (\ref{eq:lorentz}), we identify the above imbedding as $B^{s,p,r}\subset L^{r}$.
$\qed$

\section{Compactness and existence of minimizers}

\subsection{Compact imbeddings of radial subspaces.}

There are many known examples where a function space $A$ is cocompactly
imbedded into some other function space $B$, and some significant
subspace $\widetilde{A}$ of $A$ is compactly imbedded into the same
space $B$. 

For example, in the case where $A=W^{m,p}(\mathbb{R}^{N})$ and $\widetilde{A}$
is its subspace of functions supported in some fixed compact subset
of $\mathbb{R}^{N}$, then the usual subcritical Sobolev imbedding
of $A$ is cocompact, and that of $\widetilde{A}$ is compact (by
the Rellich-Kondrashov lemma). In this subsection we consider a different
case, where $\widetilde{A}$ is the subspace of all radial functions
in some function space $A$. 

We refer to \cite{Solimini} and also to Chapters 3 and 4 of \cite{Fieseler-Tintarev}
for more detailed discussions of these kind of phenomena.
\begin{thm}
\label{thm:Radial comp}Let $A$ be a reflexive Banach space which
is $\mathcal{D}_{\mathbb{Z}^{N}}$-cocompactly imbedded into $L^{p}(\mathbb{R}^{N})$
for some $p>1$\,,
Suppose that every weakly convergent sequence
in $A$ has a subsequence which converges a.e. Suppose furthermore
that $A$ is $\mathbb{R}^{N}$-shift invariant and also rotation invariant,
i.e. that \[
\left\Vert u\circ\omega\right\Vert _{A}=\left\Vert u\right\Vert _{A}\]
for each $u\in A$ and for each $\omega:\mathbb{R}^{N}\to\mathbb{R}^{N}$
which is either a shift by some element of $\mathbb{R}^{N}$ or an
element of $O(N)$. Let $A_{R}$ denote the subspace of radially symmetric
functions in $A$.

Then the imbedding of $A_{R}$ into $L^{p}(\mathbb{R}^{N})$ is compact. \end{thm}
\begin{proof}
Let $\left\{ u_{k}\right\} _{k\in\mathbb{N}}$ be an
arbitrary bounded sequence in $A_{R}$. {We
have to show that it has a subsequence which converges in norm in
$L^{p}(\mathbb{R}^{N})$. Thus, by passing if necessary to a subsequence,
we may assume without loss of generality that $u_{k}$ converges weakly
in $A$ to some limit which, again without loss of generality, may
be assumed to be $0$. This immediately implies that, for any finite
subset $E$ of $\mathbb{Z}^{N}$ and any sequence $\left\{ y_{k}\right\} _{k\in\mathbb{N}}$
in $E$, the sequence $u_{k}(\cdot-y_{k})$ tends to $0$ weakly in
$A$. If the sequence $\left\{ u_{k}\right\} _{k\in\mathbb{N}}$ is
$\mathcal{D}_{\mathbb{Z}^{N}}$-weakly convergent to $0$ in $A$
then, in view of the $\mathcal{D}_{\mathbb{Z}^{N}}$-cocompactness
of the embedding of $A$ into $L^{p}(\mathbb{R}^{N})$, the proof
is complete. Therefore we shall assume the existence of a sequence
$\left\{ y_{k}\right\} _{k\in\mathbb{N}}$ in $\mathbb{R}^{N}$ for
which $u_{k}(\cdot-y_{k})$ does not converge weakly to $0$ and show
that this leads to a contradiction. After making this assumption,
further passages to subsequences, if necessary, enable us to assume
that $u_{k}\left(\cdot-y_{k}\right)$ does have a non zero weak limit
$w\in A$ and also that \begin{equation}
\lim_{k\to\infty}\frac{y_{k}}{\left|y_{k}\right|}=z_{*}\label{eq:zstar}\end{equation}
for some point $z_{*}$ on the unit sphere of $\mathbb{R}^{N}$. The
first of these properties tells us that the sequence $\left\{ y_{k}\right\} _{k\in\mathbb{N}}$
cannot be confined to any finite subset $E$ of $\mathbb{Z}^{N}$.
So we may also assume that $\left|y_{k}\right|\to\infty$. }{\par}

{S}ince $u_{{k}}\circ\omega=u_{{k}}$,
for every $\omega\in O(N)$, it follows that $u_{k}(\cdot-\omega^{-1}y_{k})\rightharpoonup w\circ\omega$\,.
Let {$\left\{ \omega_{n}\right\} _{n\in\mathbb{N}}$
be a sequence of} elements of $O(N)$ {which
satisfy $\omega_{i}^{-1}z_{*}\ne\omega_{j}^{-1}z_{*}$ whenever $i\ne j$.}
Then {it readily follows from (\ref{eq:zstar})
that} $\left|\omega_{i}^{-1}y_{k}-\omega_{j}^{-1}y_{k}\right|\to\infty$
whenever $i\ne j$. Taking into account the continuity of the imbedding
into $L^{p}$ and the assumption about a.e.~convergence, passing
{yet again} if necessary to a subsequence, and
then applying the iterated Brezis-Lieb lemma {(Proposition
\ref{pro:IteratedBL}) with $y_{k}^{(n)}=-\omega_{n}^{-1}y_{k}$ for
each $k$ and $n$}, we conclude {from (\ref{eq:iterBL})}
that, for any $M\in\mathbb{N}$,

\[
\infty>\liminf_{k\to\infty}\int_{\mathbb{R}^{N}}|u_{k}|^{p}\ge\sum_{i=1}^{M}\int_{\mathbb{R}^{N}}|w{\circ\omega_{i}}|^{p}=M\int_{\mathbb{R}^{N}}|w|^{p}.\]

Since $M$ is an arbitrary integer, we obtain {the}
contradiction {required to complete the proof
of the theorem.} 
\end{proof}

Theorem \ref{thm:Radial comp} can be immediately combined with Theorem
\ref{thm:ournew} to give the following: 
\begin{cor}
Let $\alpha>0$, $p\in(1,\infty)$, $N<p\alpha$ and \textup{$q\in(p,p_{\alpha}^{*})$
. Then the imbedding of $W_{R}^{\alpha,p}(\mathbb{R}^{N})$ into $L^{q}(\mathbb{R}^{N})$
is compact. }
\end{cor}

Similarly, combining \ref{thm:Radial comp} with Theorem \ref{thm:Coco-Bes-Lp}
will give us:
\begin{cor}
Let $s>0,\;1<p<\infty,\: p\le q_{0}<q<p_{s}^{*}$. Then the subspace
$B_{R}^{s,p,q_{0}}(\mathbb{R}^{N})$ of radially symmetric functions
in $B^{s,p,q_{0}}(\mathbb{R}^{N})$ is compactly imbedded into $L^{q}(\mathbb{R}^{N})$. 
\end{cor}
This result may be compared to the compactness of imbeddings of radial
subspaces of Besov spaces obtained by Sickel and Skrzypczak \cite{SickelSkrzypczak}.
We remark also that some results of this kind may be obtained by interpolation
of imbeddings of radial subspaces of classical Sobolev spaces.

\subsection{Existence of minimizers.}

For $p=2$ the (fractional) Sobolev space $W^{\alpha,2}(\mathbb{R}^{N})$
is of course a Hilbert space, which is customarily denoted by $H^{\alpha}$.
One of its natural equivalent norms is given by \begin{equation}
\left\Vert f\right\Vert _{H^{\alpha}}=\left(\int_{\mathbb{R}^{N}}\left(1+\left|\xi\right|^{2}\right)^{\alpha}\left|\widehat{f}(\xi)\right|^{2}d\xi\right)^{1/2}\,.\label{eq:hn}\end{equation}

The following two theorems hold for the norm (\ref{eq:hn}) and also
for any other equivalent Hilbert norm on $H^{\alpha}$ which, like
(\ref{eq:hn}), is invariant under lattice translations. When $\alpha$
is a positive integer these theorems are well known, the first being
due to Berezstycki-P.\,-L.\,Lions \cite{BerestyckiLions} and the
second to P.\,-L.\,Lions \cite{PLL1a}. Later versions of the proofs
of their two results can be found, for example in Struwe \cite{Struwe}.
Our extensions here to the case where $\alpha$ is not an integer
are straightforward adaptations of the standard proofs.
\begin{thm}
\label{thm:autonom-min}For each $\alpha>0$ and each $q\in(2,2_{\alpha}^{*})$\,,
the infinimum \begin{equation}
\kappa:=\inf_{\|u\|_{L^{q}(\mathbb{R}^{N})}=1}\|u\|_{H^{\alpha}(\mathbb{R}^{N})}^{2}\label{eq:kappa}\end{equation}
is attained.\end{thm}
\begin{proof}
Let $\left\{ u_{k}\right\} _{k\in\mathbb{N}}$ be a minimizing sequence,
that is, $\lim_{k\to\infty}\|u_{k}\|_{H^{\alpha}(\mathbb{R}^{N})}^{2}=\kappa$
and $\|u_{k}\|_{L^{q}(\mathbb{R}^{N})}=1$. Suppose that, for every
sequence $\left\{ y_{k}\right\} _{k\in\mathbb{N}}$ in $\mathbb{Z^{N}}$,
the sequence $\left\{ u_{k}(\cdot-y_{k})\right\} _{k\in\mathbb{N}}$
converges weakly in $H^{\alpha}$ to $0$. Then $u_{k}\to0$ in $L^{q}(\mathbb{R}^{N})$,
since $H^{\alpha}(\mathbb{R}^{N})$ is $\mathcal{D}_{\mathbb{Z}^{N}}$-cocompactly
imbedded into $L^{q}(\mathbb{R}^{N})$ by Theorem \ref{thm:ournew}.
This contradicts the assumption $\|u_{k}\|_{L^{q}(\mathbb{R}^{N})}=1.$
Consequently, there exist a (possibly renamed) subsequence $\left\{ y_{k}\right\} _{k\in\mathbb{N}}$
and a function $w\in H^{\alpha}(\mathbb{R}^{N})\setminus\left\{ 0\right\} $
, such that $u_{k}(\cdot-y_{k})$ converges to $w$ weakly in $H^{\alpha}$
and also pointwise a.e. (Here we take into account that weak convergence
in $H^{\alpha}$ implies convergence locally in measure.) Furthermore,
the sequence $v_{k}=\left\{ u_{k}(\cdot-y_{k})\right\} _{k\in\mathbb{N}}$
is also a minimizing sequence. Then

\begin{eqnarray}
\kappa & = & \|v_{k}\|_{H^{\alpha}(\mathbb{R}^{N})}^{2}+o(1)=\|(v_{k}-w)+w\|_{H^{\alpha}(\mathbb{R}^{N})}^{2}+o(1)\nonumber \\
 & = & \|v_{k}-w\|_{H^{\alpha}(\mathbb{R}^{N})}^{2}+\|w\|_{H^{\alpha}(\mathbb{R}^{N})}^{2}+2\left\langle v_{k}-w,w\right\rangle +o(1)\nonumber \\
 & = & \|v_{k}-w\|_{H^{\alpha}(\mathbb{R}^{N})}^{2}+\|w\|_{H^{\alpha}(\mathbb{R}^{N})}^{2}+o(1)\,.\label{eq:mol}\end{eqnarray}

By the Brezis-Lieb lemma, 

\begin{equation}
1=\|v_{k}\|_{L^{q}(\mathbb{R}^{N})}^{q}=\|v_{k}-w\|_{L^{q}(\mathbb{R}^{N})}^{q}+\|w\|_{L^{q}(\mathbb{R}^{N})}^{q}+o(1)\,.\label{eq:ubl}\end{equation}

Since $\left\Vert f\right\Vert _{H^{\alpha}}^{2}\ge\kappa\left\Vert f\right\Vert _{L^{q}}^{2}$
for each $f\in H^{\alpha}$, we can deduce from (\ref{eq:mol}) that 

\[
\kappa\ge\kappa\|v_{k}-w\|_{L^{q}(\mathbb{R}^{N})}^{2}+\kappa\|w\|_{L^{q}(\mathbb{R}^{N})}^{2}+o(1)\,.\]
This in turn, in view of (\ref{eq:ubl}), implies that \[
\kappa\ge\kappa(1-\|w\|_{L^{q}(\mathbb{R}^{N})}^{q})^{2/q}+\kappa\|w\|_{L^{q}(\mathbb{R}^{N})}^{2}\,.\]

Since $q>2$ and $w\neq0$, the last inequality holds true only if
$\|w\|_{L^{q}(\mathbb{R}^{N})}=1$. The weak lower semicontinuity
of the norm implies that $\|w\|_{H^{\alpha}(\mathbb{R}^{N})}^{2}\le\kappa$\,.
But then $\|w\|_{H^{\alpha}(\mathbb{R}^{N})}^{2}=\kappa$, since $\kappa$
is the infimum value for such expressions. Therefore, $v_{k}$ converges
in the norm of $H^{\alpha}(\mathbb{R}^{N})$ to a minimum element
$w$.
\end{proof}

Our second theorem deduces the existence of a minimizer as a consequence
of a penalty condition.
\begin{thm}
\label{thm:penalty-min} Assume that the function $b\in C(\mathbb{R}^{N})$
has a limit at infinity and that $0<b_{\infty}:=\lim_{|x|\to\infty}b(x)<b(x)$
for all $x\in\mathbb{R}^{N}$. Then, for each $\alpha>0$ and each
\textup{$q\in(2,2_{\alpha}^{*})$\,,} the infimum \[
\tilde{\kappa}:=\inf_{\int_{\mathbb{R}^{N}}b(x)|u|^{q}dx=1}\|u\|_{H^{\alpha}(\mathbb{R}^{N})}\]
is attained. \end{thm}
\begin{proof}
Let $F(u):=\int_{\mathbb{R}^{N}}b(x)|u|^{p}dx$, $F_{0}(u):=\int_{\mathbb{R}^{N}}b_{\infty}|u|^{p}dx$
and $\Psi(u):=\int_{\mathbb{R}^{N}}(b(x)-b_{\infty})|u|^{p}dx.$ Note
that $\Psi(u)>0$ unless $u=0$. It is easy to show that$\Psi$ is
weakly continuous in $H^{\alpha}(\mathbb{R}^{N})$, by fixing an $\epsilon>0$
and dividing the domain of integration into $\lbrace b(x)-b_{\infty}\le\epsilon\rbrace$
and the bounded region $\lbrace b(x)-b_{\infty}>\epsilon\rbrace$
. Let $\lbrace u_{k}\rbrace_{k\in\mathbb{N}}$ be a minimizing sequence,
that is, $\lim_{k\to\infty}\|u_{k}\|_{H^{\alpha}(\mathbb{R}^{N})}^{2}=\kappa$
and $F(u_{k})=1$. Without loss of generality, we may assume that
$u_{k}\rightharpoonup w$ in $H^{\alpha}(\mathbb{R}^{N})$. As in
the proof of Theorem \ref{thm:autonom-min},

\begin{equation}
\tilde{\kappa}=\|u_{k}\|_{H^{\alpha}(\mathbb{R}^{N})}^{2}+o(1)=\|u_{k}-w\|_{H^{\alpha}(\mathbb{R}^{N})}^{2}+\|u\|_{H^{\alpha}(\mathbb{R}^{N})}^{2}+o(1)\,.\label{eq:one}\end{equation}
So, if we write $F(u)$ as a sum of $F_{0}(u)=\|u\|_{L^{q}(\mathbb{R}^{N})}^{q}$
and the weakly-continuous functional $\Psi(u)$, the Brezis-Lieb lemma
applied to $F_{0}$ gives us that

\begin{equation}
1=F(u_{k})=\lim_{k\to\infty}F_{0}(u_{k}-w)+F_{0}(w)+\Psi(w)\le\lim_{k\to\infty}F(u_{k}-w)+F(w),\label{eq:two}\end{equation}
where the inequality is strict unless $u_{k}\to w$ in $L^{q}(\mathbb{R}^{N})$.
Comparing (\ref{eq:one}) and (\ref{eq:two}), we obtain

\[
\tilde{\kappa}\ge\tilde{\kappa}\|u_{k}-w\|_{L^{q}(\mathbb{R}^{N})}^{2}+\tilde{\kappa}\|w\|_{L^{q}(\mathbb{R}^{N})}^{2}+o(1)\ge\tilde{\kappa}(1-F(w))^{2/q}+\tilde{\kappa}F(w)^{2/q}+o(1).\]
Since $q>2$\,, the last inequality holds true only if $F(w)=1$
or $w=0$. If, however, $w=0$, by the weak continuity of $\Psi$
and the Brezis-Lieb lemma we have $F(u_{k})=F_{0}(u_{k})$\,, which
implies that $\tilde{\kappa}\ge\kappa$\,. On the other hand, substitution
of the (renormalized) minimizer for \ref{eq:kappa} yields $\tilde{\kappa}<\kappa$\,,
a contradiction. Consequently, $F(w)=1$ and one can verify that $w$
is a minimizer by a literal repetition of the last steps of the proof
of Theorem \ref{thm:autonom-min}.
\end{proof}

\section*{Appendix A: Basics of interpolation theory and Besov spaces.}

We summarize here the basic definitions and facts about interpolation
spaces generated by the {}``real method'' (J.-L.Lions--J.Peetre
\cite{Lions-Peetre}) and by the {}``complex method'' (A.~P.~Calderón
\cite{ca}). For more details one can refer, e.g., to \cite{Adams2},
\cite{BennettSharpley}, \cite{bl}, \cite{BrudnyiKruglyak} and/or
\cite{Triebel}.

\subsection*{Banach couples}

Suppose that $A_{0}$ and $A_{1}$ are Banach spaces which are both
linear subspaces of some Hausdorff linear topological space $\mathcal{A}\ $,
and the identity maps from $A_{0}$ into $\mathcal{A}$ and from $A_{1}$
into $\mathcal{A}$ are both continuous. Then we say that $\left(A_{0},A_{1}\right)$
is a Banach couple. (It is not difficult to see that this definition
is equivalent to the seemingly more stringent definition where $\mathcal{A}$
is also required to be a Banach space.)\\

For each Banach couple $\left(A_{0},A_{1}\right)$ it is clear that
the space $A_{0}+A_{1}$ normed by $\left\Vert a\right\Vert _{A_{0}+A_{1}}:=\inf\left\{ \left\Vert a_{0}\right\Vert _{A_{0}}+\left\Vert a_{1}\right\Vert _{A_{1}}:a_{0}\in A_{0},\, a_{1}\in A_{1},\, a=a_{0}+a_{1}\right\} $
is also a Banach space.

\subsection*{The real interpolation method (J.-L.Lions--J.Peetre\cite{Lions-Peetre})}

There are several equivalent definitions of the \emph{real method
interpolation spaces} $\left(A_{0},A_{1}\right)_{\theta,q}$ of Lions--Peetre,
and here we give one of them that uses the \emph{Peetre $K$-functional}.
This is the functional defined for each fixed $t>0$ and each $a\in A_{0}+A_{1}$,
by\begin{eqnarray*}
K(t,a;A_{0},A_{1}): & = & \inf\left\{ \left\Vert a_{0}\right\Vert _{A_{0}}+t\left\Vert a_{1}\right\Vert _{A_{1}}:a_{0}\in A_{0},\, a_{1}\in A_{1},\, a=a_{0}+a_{1}\right\} .\end{eqnarray*}

Obviously, $\lbrace K(t,\cdot;A_{0},A_{1})\rbrace_{t>0}$ is a family
of equivalent norms on $A_{0}+A_{1}$\,.

For each $\theta\in(0,1)$ and $q\in[1,\infty)$, the Banach space
$\left(A_{0},A_{1}\right)_{\theta,q}$ consists of those elements
$a\in A_{0}+A_{1}$ for which the norm \[
\left\Vert a\right\Vert _{\left(A_{0},A_{1}\right)_{\theta,q}}:=\left(\int_{0}^{\infty}\left(t^{-\theta}K(t,a;A_{0},A_{1})\right)^{q}\frac{dt}{t}\right)^{1/q}.\]
is finite. This definition extends to the case $q=\infty$ with

\[
\left\Vert a\right\Vert _{\left(A_{0},A_{1}\right)_{\theta,\infty}}:=\sup_{t>0}t^{-\theta}K(t,a;A_{0},A_{1})\,.\]

Among the many known properties of these spaces, we mention the inclusions 

\begin{equation}
\left(A_{0},A_{1}\right)_{\theta,q_{0}}\subset\left(A_{0},A_{1}\right)_{\theta,q_{1}}\mbox{for all }\theta\in(0,1)\mbox{ and }1\le q_{0}\le q_{1}\le\infty\,\label{eq:monotonicity}\end{equation}
whose proof can be found, e.g., in \cite{Adams2} p.~216 Corollary
7.17, \cite{bl} p.~46 or \cite{Triebel} pp.~25--26.

\subsection*{The complex interpolation method (A.~P.~Calderón \cite{ca})}

Let $(A_{0},A_{1})$ be a Banach couple. Let $\mathcal{F}=\mathcal{F}(A_{0},A_{1})$
be the space of all functions $f$ of the complex variable $z=x+iy$
with values in $A_{0}+A_{1}$ that satisfy the following conditions:

(a) $f$ is continuous and bounded on the strip $0\le x\le1$ into
$A_{0}+A_{1}$. 

(b) $f$ is analytic from 0 < 0 < 1 into $X_{0}+X_{1}$ (i.e., the
derivative $f'(z)$ exists in $A_{0}+A_{1}$ if $0<x=Rez<1$). 

(c) $f$ is continuous on the line $x=0$ into $A_{0}$ and 

\[
\|f(iy)\|_{X_{0}}\to0\quad\text{ as }|y|\to\infty.\]

(d) $f$ is continuous on the line $x=1$ into $A_{1}$ and 

\[
\|f(1+iy)\|_{X_{1}}\to0\quad\text{ as }|y|\to\infty.\]
The space $\mathcal{F}$ is a Banach space with norm 

\[
\|f\|_{\mathcal{F}}=\max\lbrace\sup_{y\in\mathbb{R}}\|f(iy)\|_{X_{0}},\sup_{y\in\mathbb{R}}\|f(1+iy)\|_{X_{1}}\rbrace.\]
Given a real number $\theta$ in the interval $(0,1)$, we define

\[
A_{\theta}=[A_{0},A_{1}]_{\theta}=\lbrace u\in A_{0}+A_{1}:\; u=f(\theta)\text{ for some }f\in\mathcal{F}\rbrace.\]

The spaces $A_{\theta}$ are called\emph{ complex interpolation spaces
}between $A_{0}$ and $A_{1}$; they are Banach spaces with respective
norms

\[
\|u\|_{A_{\theta}}=\inf\lbrace\|f\|_{\mathcal{F}}:f(\theta)=u\rbrace.\]

\subsection*{The basic interpolation theorem}
\begin{thm}
\label{thm:fund interp}Let $(A_{0},A_{1})$ and $(B_{0,}B_{1})$
be two Banach couples. If there exists a linear operator $T:A_{0}+A_{1}\to B_{0}+B_{1}$,
which is continuous as a map from $A_{j}$ to $B_{j}$ for $j=0,1$,
then, for each $\theta\in(0,1)$ and each $p\in[1,\infty]$, \begin{equation}
\|T\|_{(A_{0},A_{1})_{\theta,p}\to(B_{0},B_{1})_{\theta,p}}\le\|T\|_{A_{0}\to B_{0}}^{\theta}\|T\|_{A_{1}\to B_{1}}^{1-\theta}\label{eq:interp imbeddings}\end{equation}
and \begin{equation}
\|T\|_{[A_{0},A_{1}]_{\theta}\to[B_{0},B_{1}]_{\theta}}\le\|T\|_{A_{0}\to B_{0}}^{\theta}\|T\|_{A_{1}\to B_{1}}^{1-\theta}\,.\label{eq:cx-interp imbed}\end{equation}

\end{thm}
(See, e.g., \cite{Adams2} pp.~220--221, \cite{bl} pp.~40--41 and
p.~88.)

\subsection*{The special case $A_{0}=A_{1}$ and some special norm estimates}

In the case where $A_{0}=A_{1}$ isometrically we obtain that $\left[A_{0},A_{0}\right]_{\theta}=A_{0}$
for each $\theta\in(0,1)$, with \begin{equation}
\left\Vert a\right\Vert _{A_{0}}=\left\Vert a\right\Vert _{\left[A_{0},A_{0}\right]_{\theta}}\label{eq:aisom}\end{equation}
and $\left(A_{0},A_{0}\right)_{\theta,p}=A_{0}$ for each $\theta\in(0,1)$
and $p\in[1,\infty]$, with \begin{equation}
\left\Vert a\right\Vert _{\left(A_{0},A_{0}\right)_{\theta,p}}=c_{\theta,p}\left\Vert a\right\Vert _{A_{0}}\,.\label{eq:rmis}\end{equation}
where the constant $c_{\theta,p}$ is given by \[
c_{\theta,\infty}=1\ \mbox{and }c_{\theta,p}=\left(\frac{1}{\theta\left(1-\theta\right)p}\right)^{1/p}\,.\]
The proof of (\ref{eq:aisom}) is straightforward. For {}``$\le$''
one can use the Phragmen-Lindelof theorem for analytic $A_{0}$ valued
functions on the strip $0\le\theta\le1$. Then for {}``$\ge$''
one can use the function $f\in\mathcal{F}\left(A_{0},A_{0}\right)$
defined by $f(z)=e^{\delta(z-\theta)^{2}}a$ where $\delta$ is an
arbitrarily small positive number. The proof of (\ref{eq:rmis}) follows
immediately from the fact that $K(t,a;A_{0},A_{0})=\min\left\{ 1,t\right\} \left\Vert a\right\Vert _{A_{0}}$. 

We will need some standard estimates for the norms $\left\Vert a\right\Vert _{\left[A_{0},A_{1}\right]_{\theta}}$
and $\left\Vert a\right\Vert _{\left(A_{0},A_{1}\right)_{\theta,p}}$
in the case where $a\in A_{0}\cap A_{1}$. These are \begin{equation}
\left\Vert a\right\Vert _{\left[A_{0},A_{1}\right]_{\theta}}\le\left\Vert a\right\Vert _{A_{0}}^{1-\theta}\left\Vert a\right\Vert _{A_{1}}^{\theta}\label{eq:cmq}\end{equation}
 and \begin{equation}
\left\Vert a\right\Vert _{\left(A_{0},A_{1}\right)_{\theta,p}}\le c_{\theta,p}\left\Vert a\right\Vert _{A_{0}}^{1-\theta}\left\Vert a\right\Vert _{A_{1}}^{\theta}\,.\label{eq:drb}\end{equation}
We can obtain (\ref{eq:cmq}) from an easy exercise using the function
$f(z)=\frac{e^{\delta(z-\theta)^{2}}}{\left\Vert a\right\Vert _{A_{0}}^{1-z}\left\Vert a\right\Vert _{A_{1}}^{z}}a$
for arbitrarily small $\delta>0$ and an estimate very similar to
(\ref{eq:drb}) is implicit in pp.~49--50 of \cite{bl}. But we can
also prove both (\ref{eq:cmq}) and (\ref{eq:drb}) simultaneously,
as follows. Let $X$ be either $\left(A_{0},A_{1}\right)_{\theta,q}$
or $\left[A_{0},A_{1}\right]_{\theta}$ and consider the linear operator
$L:\mathbb{C}\to A_{0}\cap A_{1}$ defined by $Lz=za$ for each $z\in\mathbb{C}$.
Then $\left\Vert L\right\Vert _{\mathbb{C}\to A_{j}}=\left\Vert a\right\Vert _{A_{j}}$
for $j=0,1$ and so, by Theorem \ref{thm:fund interp} and (\ref{eq:aisom})
and (\ref{eq:rmis}), we have \[
\left\Vert a\right\Vert _{\left[A_{0},A_{1}\right]_{\theta}}=\left\Vert L1\right\Vert _{\left[A_{0},A_{1}\right]_{\theta}}\le\left\Vert a\right\Vert _{A_{0}}^{1-\theta}\left\Vert a\right\Vert _{A_{1}}^{\theta}\left\Vert 1\right\Vert _{[\mathbb{C},\mathbb{C}]_{\theta}}=\left\Vert a\right\Vert _{A_{0}}^{1-\theta}\left\Vert a\right\Vert _{A_{1}}^{\theta}\]
and \[
\left\Vert a\right\Vert _{\left(A_{0},A_{1}\right)_{\theta,p}}=\left\Vert L1\right\Vert _{\left(A_{0},A_{1}\right)_{\theta,p}}\le\left\Vert a\right\Vert _{A_{0}}^{1-\theta}\left\Vert a\right\Vert _{A_{1}}^{\theta}\left\Vert 1\right\Vert _{\left(\mathbb{C},\mathbb{C}\right)_{\theta,p}}=c_{\theta,p}\left\Vert a\right\Vert _{A_{0}}^{1-\theta}\left\Vert a\right\Vert _{A_{1}}^{\theta}\,.\]

\subsection*{Interpolation formulæ for $L^{p}$ spaces}

When applied to a couple of $L^{p}$ spaces on the same underlying
measure space, both the complex and the real methods (the latter for
a suitable choice of the second parameter) yield an $L^{p}$ space
with an intermediate exponent: 

\begin{equation}
\left(L^{p_{0}},L^{p_{1}}\right)_{\theta,p}=\left[L^{p_{0}},L^{p_{1}}\right]_{\theta}=L^{p}\mbox{ for all }1\le p_{0}<p_{1}\le\infty\mbox{ and }\theta\in(0,1),\label{eq:lorentz}\end{equation}
where $p$ is given by $\frac{1}{p}=\frac{1-\theta}{p_{0}}+\frac{\theta}{p_{1}}$.
(See e.g., \cite{Adams2} Corollary 7.27 p.~226 and Example 7.56
on pp.~249--250.)

\subsection*{Fractional Sobolev spaces and Besov spaces}

There are several equivalent definitions of 
{these}
spaces. For our
purposes here it will be convenient to define them via complex or
real interpolation of Sobolev spaces and $L^{p}$spaces.

Fractional Sobolev spaces can be equivalently defined (see e.g.~\cite{Adams2}
p.~250) by \begin{equation}
W^{s,p}(\mathbb{R}^{N})=[W^{m,p}(\mathbb{R}^{N}),L^{p}(\mathbb{R}^{N})]_{s/m},\; m\in\mathbb{N},p\in(1,\infty),0<s<m\,.\label{eq:fss}\end{equation}
Note that all choices of $m$ as above give the same space.

For each $s\in(0,\infty)$, $p\in(1,\infty)$ and $q\in[1,\infty]$
the Besov space $B^{s,p,q}(\mathbb{R}^{N})$ (see \cite{Triebel}
p.~186, see also p.~230 of \cite{Adams2} as well as pp.~139--145
of \cite{bl}) can be defined by the formula: 

\begin{equation}
B^{s,p,q}(\mathbb{R}^{N})=\left(W^{s_{0},p}(\mathbb{R}^{N}),W^{s_{1},p}(\mathbb{R}^{N})\right)_{\theta,q}\,,0\le s_{0}<s<s_{1}\mbox{ and }\theta=\frac{s-s_{0}}{s_{1}-s_{0}}\,.\label{eq:besovbi}\end{equation}
A commonly used version of this definition uses only integer values
of $s_{0}$ and $s_{1}$. Analogously to the previous definition,
all choices of $s_{0}$ and $s_{1}$ as above give the same space,
to within equivalence of norms. 

The Besov spaces satisfy the following continuous imbeddings (Jawerth
\cite{Jawerth}, see also \cite{Triebel} Theorem 2.8.1 p.~203): 

\begin{equation}
B^{s_{0},p_{0},q}(\mathbb{R}^{N})\subset B^{s,p,q}(\mathbb{R}^{N}),\quad1<p_{0}<p<\infty,\;1\le q\le\infty,\; s_{0}-s\ge n/p_{0}-n/p\,.\label{eq:be2}\end{equation}
These imbeddings can also be obtained from \cite{Adams2} Theorem
7.34 p.~231 by applying the reiteration formula for real interpolation
spaces.

The Besov spaces also admit the following continuous imbeddings into
$L^{p}$ spaces:

\begin{equation}
B^{\{s,p,q\}}(\mathbb{R}^{N})\subset L^{q}(\mathbb{R}^{N}),\: s>0,\;1<p<\infty,\: p\le q<p_{s}^{*}\label{eq:be:lp}\end{equation}

\section*{Appendix B: The iterated Brezis-Lieb lemma}

The following proposition evaluates the $L^{p}$-norms of sequences
given by sums of terms with asymptotically disjoint supports. Although
it and similar results have appeared elsewhere in literature, for
the reader's convenience we explicitly recall its proof, which is
an easy corollary of the well known Brezis-Lieb lemma \cite{Brezis-Lieb}.
\begin{prop}
\label{pro:IteratedBL}
Suppose that $1\le p<\infty$. Let $y_{k}^{(n)}$ be a point in $\mathbb{R}^{N}$
for each $k$ and $n$ in $\mathbb{N}$. Suppose that $\lim_{k\to\infty}|y_{k}^{(m)}-y_{k}^{(n)}|=+\infty$
for each fixed $m$ and $n$ with $m\neq n$. Let $u_{k}\in L^{p}(\mathbb{R}^{n})$
be a bounded sequence such that, for each $n\in\mathbb{N}$, the sequence
$u_{k}(\cdot+y_{k}^{(n)})$ converges weakly and almost everywhere
to a function which we will denote by $w^{(n)}$. Then, for every
$M\in\mathbb{N}$,

\begin{equation}
\int_{\mathbb{R}^{N}}|u_{k}|^{p}-\sum_{n=1}^{M}\int_{\mathbb{R}^{N}}|w^{(n)}|^{p}-\int_{\mathbb{R}^{N}}\left|u_{k}-\sum_{n=1}^{M}w^{(n)}(\cdot-y_{k}^{(n)})\right|^{p}\to0.\label{eq:iterBL}\end{equation}
\end{prop}
\begin{proof}

{It of course suffices to show that every subsequence
of the left side of (\ref{eq:iterBL}) has itself a subsequence which
tends to $0$ as $k\to\infty$. This means that we can assume, without
loss of generality, that the sequences $\left\{ y_{k}^{(n)}\right\} _{k\in\mathbb{N}}$
satisfy the additional condition $|y_{k+1}^{(m)}-y_{k+1}^{(n)}|>2|y_{k}^{(m)}-y_{k}^{(n)}|$
for each $k\in\mathbb{N}$ and each $m,n\in\left\{ 1,2,...,M\right\} $
with $m\ne n$. This guarantees that $\lim_{k\to\infty}f(x+y_{k}^{(m)}-y_{k}^{(n)})=0$
for a.e. $x\in\mathbb{R}^{N}$ for each $f\in L^{p}(\mathbb{R}^{N})$,
and, in particular, whenever $f$ is any one of the functions $w^{(j)}$,
$j\in\mathbb{N}$.}

We use induction. For $M=1$, the statement is immediate from the
Brezis-Lieb lemma for the sequence $u_{k}(\cdot+y_{k}^{(1)})$ whose
weak and a.e. limit is $w^{(1)}$. Assume (\ref{eq:iterBL}) is true
for 
{$M{=m}$}
and let us show that it is true for $M=m+1$. Let

\[
v_{k}^{(m)}=u_{k}-\sum_{n=1}^{m}w^{(n)}(\cdot-y_{k}^{(n)}).\]

Applying the Brezis-Lieb lemma to the sequence $v_{k}^{(m)}(\cdot+y_{k}^{(m+1)})$
whose weak and a.e. limit is $w^{(m+1)}$, we obtain from (\ref{eq:iterBL})
the following:

\[
0=\lim\left[\int_{\mathbb{R}^{N}}|u_{k}|^{p}-\sum_{n=1}^{m}\int_{\mathbb{R}^{N}}|w^{(n)}|^{p}-\int_{\mathbb{R}^{N}}\left|v_{k}^{(m)}(\cdot+y_{k}^{(n)})\right|^{p}\right]\]

\[
=\lim\left[\int_{\mathbb{R}^{N}}|u_{k}|^{p}-\sum_{n=1}^{m}\int_{\mathbb{R}^{N}}|w^{(n)}|^{p}-\int_{\mathbb{R}^{N}}|w^{(m+1)}|^{p}-\int_{\mathbb{R}^{N}}\left|v_{k}^{(m+1)}(\cdot+y_{k}^{(n)})\right|^{p}\right]\]

which immediately gives (\ref{eq:iterBL}) for $M=m+1$. \end{proof}

\end{document}